\documentclass[3p,12pt]{elsarticle}

\usepackage{amsmath,amsthm}
\usepackage{amsthm,bm}
\usepackage{amssymb}
\usepackage{amsfonts}
\usepackage{dsfont}
\usepackage{graphicx}
\usepackage{enumerate}
\usepackage[dvipsnames]{xcolor}
\usepackage{bm}
\usepackage{float}
\usepackage{graphicx}

\usepackage[normalem]{ulem}

\theoremstyle{plain}
\newtheorem{theorem}{Theorem}

\newtheorem{lemma}{Lemma}
\newtheorem{corollary}{Corollary}

\theoremstyle{definition}
\newtheorem{definition}{Definition}

\newtheorem{example}{Example}

\newtheorem{problem}{Problem}

\theoremstyle{remark}
\newtheorem{remark}{Remark}

\def\prank{\mbox{\rm prank}}

\def\mr+{\mbox{$\text{mr}_{+}$}}
\DeclareMathOperator*{\argmin}{arg\,min}

\usepackage{euscript}

\newcommand{\mR}{\mathbb{R}}

\newcommand{\script}[1]{\EuScript{#1}}

\begin{document}
\begin{frontmatter}
\title{Low Phase-Rank Approximation
}

\author[HKUST]{Di Zhao\footnote{Corresponding author.}}\ead{dzhaoaa@connect.ust.hk}  
\author[HKUST]{Axel Ringh}\ead{eeringh@ust.hk} 
\author[HKUST]{Li Qiu}\ead{eeqiu@ust.hk}  
\author[HKU]{Sei Zhen Khong}\ead{szkhongwork@gmail.com}

\address[HKUST]{Department of Electronic and Computer Engineering, The Hong Kong University of Science and Technology, Clear Water Bay, Kowloon, Hong Kong, China}  
\address[HKU]{Independent researcher}  

\begin{keyword}                           
	Matrix phase, phase-rank, low phase-rank approximation, geometric mean, geodesic distance
	
	\vskip 3pt
	
	\MSC[2020] 15A45, 15A60, 15B48, 47A58, 53C22
\end{keyword}    

\begin{abstract}
In this paper, we propose and solve a low phase-rank approximation problem, which serves as a counterpart to the well-known low-rank approximation problem and the Schmidt-Mirsky theorem. More specifically, a nonzero complex number can be specified by its gain and phase, and while it is generally accepted that the gains of a matrix may be defined by its singular values, there is no widely accepted definition for its phases.
In this work, we consider sectorial matrices, whose numerical ranges do not contain the origin, and adopt the canonical angles of such matrices as their phases. 
Similarly to the rank of a matrix defined to be the number of its nonzero singular values, we define the phase-rank of a sectorial matrix as the number of its nonzero phases. While a low-rank approximation problem is associated with matrix arithmetic means, as a natural parallel we formulate a low phase-rank approximation problem using matrix geometric means to measure the approximation error. A characterization of the solutions to the proposed problem is then obtained, when both the objective matrix and the approximant are restricted to be positive-imaginary. Moreover, the obtained solution has the same flavor as the Schmidt-Mirsky theorem on low-rank approximation problems. In addition, we provide an alternative formulation of the low phase-rank approximation problem using geodesic distances between sectorial matrices. The two formulations give rise to the exact same set of solutions when the involved matrices are additionally assumed to be unitary.

\end{abstract}
\end{frontmatter}

\setcounter{footnote}{0}


\section{Introduction}

Many well-known concepts, such as degree, order, dimensionality and so on, can be interpreted as the rank of a certain matrix \cite{recht2010guaranteed}. A real-world system is often complicated, resulting in a mathematical model with high complexity, namely, a high-dimensional matrix with a large rank. Consequently, low-rank approximations to such models are of critical importance in understanding and analysing the system and have been extensively investigated over decades \cite{recht2010guaranteed,stewart1990matrix,ivan2008structured,ivan2014lowrank}. Among those celebrated results on low-rank approximation problems, the Schmidt-Mirsky theorem \cite[Chap.~IV]{stewart1990matrix}, \cite[Chap.~2]{ivan2014lowrank}, which gives rise to an analytic solution to the unstructured low-rank approximation, has been one of the most fundamental. In this study, we propose a parallel to the concept of rank --- the \emph{p}hase-\emph{rank} (prank). In particular, we formulate a low-prank approximation problem and obtain the counterpart result to the Schmidt-Mirsky theorem.

The rank of a matrix is defined as the number of its nonzero singular values.  While singular values are well accepted as the gains of a matrix in different directions along the corresponding singular vectors, the (canonical) phases of a matrix, though not widely utilized, were originally introduced in \cite{Johnson1974209,Johnson2001spectral} and further developed in \cite{WANG2020PhaseMath}. Such phases are defined for a special family of square matrices, called sectorial matrices, whose numerical ranges do not contain the origin. The phases share many parallel properties compared with the singular values in various aspects as detailed in \cite{Johnson2001spectral,WANG2020PhaseMath}. 

Based on the definition of canonical phases, we propose the concept of phase-rank, defined as the number of the nonzero phases of a sectorial matrix. As phases are the counterpart to singular values, the prank is a natural counterpart to rank. Motivated by the low-rank approximation problem, we formulate in parallel a low-prank approximation problem, aiming at solving a matrix approximation problem in phases subject to low-prank constraints. Moreover, we obtain the counterpart result to the Schmidt-Mirsky theorem that characterizes a set of analytic solutions to the low-prank optimization. The approximation error in low-rank approximation is usually measured in terms of the norm on subtraction between matrices, namely, $\|A-E\|$. As a parallel, we measure the error in low-prank approximation using phases of some division operations, namely, $\phi(E^{-1}AE^{-1})$. Such operations can also be interpreted via the notions of arithmetic and geometric means \cite{drury2015principal}, respectively, giving rise to another parallel between the gain and phase. Furthermore, we introduce an alternative formulation for the low-prank approximation, which is based on a family of geodesic distances between sectorial matrices. The two formulations give rise to the exact same set of solutions when the involved matrices are additionally assumed to be unitary.

A sectorial matrix with zero prank is positive definite, whereby the prank can be regarded as a quantized measure on how close a matrix is to positive definiteness. This serves as a counterpart to the fact that the rank measures how close a matrix is to zero. Consequently, a low-prank approximation can be interpreted as finding an approximant, which is as close to positive definiteness as possible, to the objective matrix. 
Moreover, as the canonical phases of matrices have been exploited in \cite{chen2019phase} as a useful notion in handling systems and control problems, the developed results on low-prank approximation may also be applicable to such problems. In terms of computation, both the low-rank and low-prank approximation problems are nonconvex. In general, a low-prank approximation problem is connected to but cannot be reformulated as a low-rank approximation, which is detailed in Section~\ref{sec:LPA_GeoMean}. 

The rest of the paper is organized as follows. In Section~\ref{sec:pre}, basic notation and preliminary results are introduced. In Section~\ref{sec:LPA_GeoMean}, a low-prank approximation problem is formulated based on the geometric mean of matrices and a theorem characterizing its solutions is obtained. In addition, an alternative problem formulation based on a family of geodesic distances is proposed in Section~\ref{sec:geodesic}, to which the set of optimal solutions is characterized accordingly. Finally, the paper is concluded in Section~\ref{sec:conc}.

\section{Preliminaries}\label{sec:pre}
\subsection{Basic Notation}
Let $\mathbb{F} = \mathbb{R}$ or $\mathbb{C}$ be the real or complex field,  and ${\mathbb F}^n$ be the linear space of $n$-tuples of $\mathbb{F}$ over the field $\mathbb{F}$. The Euclidean norm of a vector $x\in{\mathbb F}^n$ is denoted by $\|x\|_2$. A nonzero complex number has its polar form $c=|c|e^{j\angle c}$, where $|c|>0$ is called its length or magnitude and $\angle c\in(-\pi,\pi]$ is called its angle or phase. 

For a matrix $A\in{\mathbb C}^{n \times n}$, its range is denoted by $\mathcal{R}(A)$, its kernel by $\mathcal{K}(A)$, its complex conjugate by $A^*$, and its singular values by {$\sigma(A) = [\sigma_k(A)]_{k=1}^n$, where}  $\sigma_k(A)$, $k=1,2,\dots,n$ {are ordered} in a nonincreasing order. 
An identity matrix is denoted by $I_n\in\mathbb{C}^{n\times n}$. A matrix $U\in{\mathbb C}^{n\times n}$ is said to be unitary, denoted by $U\in\mathcal{U}_n$, if $U^*U=UU^*=I_n$.

\subsection{The Canonical Phases of a Matrix}
A matrix $A\in\mathbb{C}^{n\times n}$ is said to be \emph{sectorial} if its numerical range
$$\mathcal{W}(A):=\{x^*Ax \mid x\in\mathbb{C}^n,~\|x\|_2=1\}$$ does not contain $0$. {A} sectorial matrix $A$ is
congruent to a diagonal unitary matrix, {where the latter} is unique up to a
permutation \cite{horn1959eigenvalues,Johnson1974209,zhang2015matrix,WANG2020PhaseMath}, i.e., $A$ admits the following sectorial decomposition
\begin{align}\label{eq:sec_dec_matrix}
A=T^*DT,
\end{align}
where $T$ is nonsingular and $D$ is diagonal and unitary. As $D$ is unique up to a permutation, the (canonical) phases of $A$, denoted by
$$\bar{\phi}(A):=\phi_1(A)\geq \cdots \geq \phi_n(A)=:\underline{\phi}(A),$$
are defined as the phases of {the} eigenvalues (which are diagonal elements) of $D$, and by convention taking values so that $\bar{\phi}(A)-\underline{\phi}(A)<\pi$. Moreover, the phase center of $A$ is defined as
$$\gamma(A):=\frac{\bar{\phi}(A)+\underline{\phi}(A)}{2}\in(-\pi,\pi].$$
For notational convenience, denote by {$\phi(A)=[\phi_k(A)]_{k=1}^n$} and $\phi_{n+1}(A)=0$.


Let $\alpha,\beta\in\mathbb{R}$ with $0<\beta-\alpha\leq\pi$ and $(\alpha+\beta)/{2}\in(-\pi,\pi]$. A matrix $A\in\mathbb{C}^{n\times n}$ is said to be sectorial in $[\alpha,\beta]$ (similarly defined for $[\alpha,\beta)$, $(\alpha,\beta]$ and $(\alpha,\beta)$), denoted by $A\in\mathcal{C}[\alpha,\beta]$,  if $A$ is sectorial and  $[\underline{\phi}(A),\bar{\phi}(A)]\subset[\alpha,\beta]$.

In particular, a sectorial matrix $A$ is called
positive-real (accretive) if $A\in\mathcal{C}[-\pi/2,\pi/2]$, positive-imaginary if $A\in\mathcal{C}[0,\pi)$, and negative-imaginary if $A\in\mathcal{C}(-\pi,0]$.

\subsection{Low-Rank Approximation}
A function $\Phi:~\mathbb{R}^n \to \mathbb{R}$ is said to be a symmetric gauge function \cite{mirsky1960symmetricgaugefunction}, \cite[Chap.~II]{stewart1990matrix} if it satisfies the following conditions: 1) $x\neq0$ $\Rightarrow$ $\Phi(x)>0$; 2) $\Phi(\rho x)=|\rho|\Phi(x)$; 3) $\Phi(x+y)\leq \Phi(x) + \Phi(y)$; 4) for any permutation matrix $P$ we have $\Phi(Px)=\Phi(x)$; 5) $\Phi(|x|)=\Phi(x)$.

A norm $\|\cdot\|$ on $\mathbb{C}^{n \times n}$ is said to be unitarily invariant \cite[Chap.~3]{stewart1990matrix} if for all {$A \in \mathbb{C}^{n\times n}$ and all} unitary matrices $U,V\in\mathbb{C}^{n\times n}$, it holds that
$$\|U^*AV\|=\|A\|.$$ 
The set of all unitarily invariant norms can be characterized by the symmetric gauge functions. More precisely, let $\Phi$ be a symmetric gauge function and define
$$\|A\|_\Phi:= \Phi(\sigma(A)) =\Phi(\sigma_1(A),\dots,\sigma_{n}(A)).$$ 
Then the norm $\|\cdot\|$ is unitarily invariant on $\mathbb{C}^{n\times n}$ if and only if there is a symmetric gauge function $\Phi$ on $\mathbb{R}^n$ such that $\|A\|=\|A\|_\Phi$ for all $A\in\mathbb{C}^{n\times n}$ \cite[Thm.~3.6]{stewart1990matrix}.


A singular value decomposition (SVD) of $A\in{\mathbb C}^{n \times n}$ is denoted by
$$A=USV^*=\sum_{i=1}^{n}\sigma_i(A)u_iv_i^*,$$
where $S=\text{diag}(\sigma_1(A),\sigma_2(A),\dots,\sigma_{n}(A))$,
and  $U=\begin{bmatrix}u_1,u_2,\dots,u_n\end{bmatrix}$, $V=\begin{bmatrix}v_1,v_2,\dots,v_n\end{bmatrix}$ are unitary. The set of all $r$-truncations of $A$ by SVD is defined as
\begin{align}\label{eq:truncation_rank}{\mathcal{S}_g(A,r)}:=\left\{\sum_{i=1}^{r}\sigma_i(A)u_iv_i^*\Bigg|~\sum_{i=1}^{n}\sigma_i(A)u_iv_i^*~\text{is an SVD of}~A\right\},\end{align}
{In particular, note that $\mathcal{S}_g(A,r)$ is a singleton if and only if $\sigma_r(A) > \sigma_{r+1}(A)$.}

Given $A\in\mathbb{C}^{n \times n}$, a low-rank matrix approximation problem \cite{eckart1936approximation} aims at finding an {approximation $E$ of the object $A$} such that {the rank of $E$ is bounded by some given $r\geq 0$ and that the distance between $E$ and $A$ is as small as possible.} A standard and well-studied form of such a problem is formulated as follows. 
\begin{problem}[Low-Rank Approximation]\label{def:lowRankApprox}
	Given $A\in\mathbb{C}^{n \times n}$ and $0\leq r \leq n$, determine
	\begin{equation}\label{eq:low_rank_approx_def}
	{\hat{E}}=\argmin_E\{\|A-E\| \mid \text{\rm rank}(E)\leq r\}.
	\end{equation}
\end{problem}

 A well established-result for the above problem, called the Schmidt-Mirsky theorem \cite[Chap.~IV]{stewart1990matrix}, can be stated in the following form. 

\begin{lemma}[Schmidt-Mirsky Theorem]\label{lem:Schmidt_Mirsky}
	Let $\|\cdot\|$ be a unitarily invariant norm on $\mathbb{C}^{n \times n}$ and $0\leq r\leq n$. Then for $A\in\mathbb{C}^{n \times n}$, we have
	$$\min_{E}\{\|A-E\| |~ \text{\rm rank}(E)\leq r\}=\|\text{\rm diag}(0,\dots,0,\sigma_{r+1}(A),\dots,\sigma_n(A))\|,$$
	where the optimum is attained on each $\hat{E}\in \mathcal{S}_g(A,r)$. 
\end{lemma}

\section{Low Phase-Rank Approximation via Geometric Means}\label{sec:LPA_GeoMean}
In this section, we start with the definition of phase-rank (prank), as a parallel concept to rank, and then formulate a low-prank approximation problem with the approximation error induced by the matricial geometric means. {Finally,  a characterization of solutions to such a problem is obtained, as a parallel to the famous Schmidt-Mirsky theorem in Lemma~\ref{lem:Schmidt_Mirsky}.}
\subsection{Definition of Phase-Rank}
It is well known that the rank of a square matrix can be defined as the number of nonzero singular values, counting multiplicities.
Motivated by this understanding and the parallel between gains and phases, we define the \emph{phase rank} of a sectorial matrix as follows.

\begin{definition}
For a sectorial matrix $A\in\mathbb{C}^{n\times n}$, we define its phase rank as
\[
\prank(A) := \text{the number of nonzero canonical phases of } A.
\]
\end{definition}
Note that $\prank(A)\in[0,n]$, and that if $0\notin \angle \mathcal{W}(A)$, i.e., when $\mathcal{W}(A) \cap \mathbb{R}_+ = \emptyset$, then $A$ is guaranteed to have full phase-rank.
We summarize some of the comparisons on the properties between the rank and prank in Table~1.

\begin{table}[]
	\begin{tabular}{l|l|l}
		& Rank                                                                    & Prank                                                                      \\ \hline
		Applicable matrices & All complex-valued matrices & Sectorial matrices\\ \hline
		Definition & \begin{tabular}[c]{@{}l@{}}No.~of nonzero singular values\end{tabular} & \begin{tabular}[c]{@{}l@{}}No.~of nonzero  canonical phases\end{tabular}        \\ \hline
		Zero value   & Zero matrix                                                             & \begin{tabular}[c]{@{}l@{}}Positive definite matrix\end{tabular} \\ \hline
		Small value  & \begin{tabular}[c]{@{}l@{}}Singular matrices \&  \\have advantage in storage \end{tabular}    & \begin{tabular}[c]{@{}l@{}}Nonsingular matrices  \&  \\``close'' to positive-definiteness\end{tabular}
	\end{tabular}
	\caption{Comparisons on the properties between rank and prank. }
\end{table}

\subsection{Problem Formulation on Low-Prank Approximation}
In order to gain some insights {into how} {to formulate the low-prank approximation problem}, we first interpret the well-known low-rank approximation problem via the matricial arithmetic mean. To this end,
let $M,N\in\mathbb{C}^{n \times n}$, and {define} their arithmetic mean
as%
\footnote{We use the same notation for the arithmetic mean as in the seminal paper by Kubo and Ando \cite{kubo1980means}.}
$${M \nabla N}:=\frac{M+N}{2}.$$
Using the arithmetic mean, the low-rank approximation problem in Problem~\ref{def:lowRankApprox} can be restated as follows: given a symmetric gauge function $\Phi:~\mathbb{R}^n \to \mathbb{R}$, a matrix $A\in\mathbb{C}^{n \times n}$, and $0\leq r \leq n$, determine
\begin{align}\label{eq:low_rank_restate}
\hat{X}=\argmin_{X}\{\Phi(\sigma(X)) \mid \text{rank}((-X){\nabla} A)\leq r\}.
\end{align}
To see this, simply make the substitution $E = A - X = 2(-X){\nabla} A$ in Problem~\ref{def:lowRankApprox}. Consequently, the optimal low-rank approximation is given by $A-\hat{X}$ where $\hat{X}$ solves \eqref{eq:low_rank_restate}.

Keeping the arithmetic mean for low-rank approximation in mind, next we formulate a low-prank approximation problem. A multiplicative operation of two matrices naturally results in an additive operation on their canonical phases; e.g., for $a,b\in\mathbb{C}$, $\angle ab=\angle a+\angle b ~\text{mod} ~2\pi$. Therefore, instead of considering the arithmetic mean of two matrices as in the low-rank approximation, in the low-prank approximation we turn to the geometric mean to measure the approximation error.

{More specifically, the geometric mean between two strictly positive-real matrices $M,N \in \mathcal{C}(-\pi/2,\pi/2)$, denoted by $M\#N$, was defined in \cite{drury2015principal} as
$$(M\#N)^{-1}={\frac{2}{\pi}}\int_0^\infty(tM+t^{-1}N)^{-1}\frac{dt}{t},$$
as a generalization of the geometric mean for positive definite matrices, see, e.g., \cite[Chp.~4 and 6]{bhatia2007positive}.
Moreover, in \cite{drury2015principal} it was shown that $M\#N \in \mathcal{C}(-\pi/2,\pi/2)$, that $G = M\#N$ is the unique solution to the equation $GM^{-1}G=N$, and that the geometric mean can also be expressed as
$$M\#N=M^{\frac{1}{2}}(M^{-\frac{1}{2}}NM^{-\frac{1}{2}})^{\frac{1}{2}}M^{\frac{1}{2}}=N^{\frac{1}{2}}(N^{-\frac{1}{2}}MN^{-\frac{1}{2}})^{\frac{1}{2}}N^{\frac{1}{2}},$$
where $(\cdot)^{\frac{1}{2}}$ denotes the principal branch of the matrix square root, see, e.g., \cite[Sec.~1.7]{higham2008functions}.
Next, for $M,N \in \mathcal{C}[\alpha-\pi/2,\alpha+\pi/2)$ it is easily verified that we can define the geometric mean analogously: either by defining it as $e^{-i\beta}((e^{i\beta}M) \# (e^{i \beta}N))$ for any $\beta$ such that $e^{i \beta}M, e^{i \beta}N \in \mathcal{C}(-\pi/2,\pi/2)$, or by using another branch of the matrix square root, cf.~\cite{higham2008functions}.}

On the other hand, note that every unitarily invariant norm on a matrix can be obtained by a symmetric gauge function applied to its singular values. As for the phase counterpart, a natural choice for the objective function in a low-prank approximation would be a symmetric gauge function applied to the canonical phases. 

Based on the above observations, as a counterpart to the low-rank approximation problem in \eqref{eq:low_rank_restate} we propose the following formulation for a low-prank approximation problem:
given a symmetric gauge function $\Phi:~\mathbb{R}^n \to \mathbb{R}$, a sectorial $A\in\mathbb{C}^{n\times n}$ and $0\leq r \leq n$, determine
\begin{align}\label{eq:low_prank0}
{\hspace{-5pt}\hat{X}=\argmin_X\left\{\Phi(\phi(X)) \mid X^{-1},A\in\mathcal{C}[\alpha-\frac{\pi}{2},\alpha+\frac{\pi}{2})~\text{for some}~\alpha,~\prank (X^{-1}\#A)\leq r \right\}.}
\end{align}
\begin{remark}
	The objective function used in the low-prank approximation problem \eqref{eq:low_prank0} is invariant under congruence, i.e., $\Phi(\phi(X))=\Phi(\phi(T^*XT))$ for every sectorial $X$ and  nonsingular $T$. This can be seen as the counterpart to $\Phi(\sigma(Y))$, used as the objective function in the low-rank approximation problem \eqref{eq:low_rank_approx_def}, defining a unitarily invariant norm on~$Y$.
\end{remark}
Finally, by substituting $E = X^{-1}\#A$ in \eqref{eq:low_prank0}, i.e., substituting $X = E^{-1}AE^{-1}$, and by relaxing some phase-range constraints\footnote{It is noteworthy that in \eqref{eq:low_prank0}, $X^{-1}$ is required to be sectorial in the same half-plane as $A$, while such a constraint is relaxed in Problem~\ref{def:low_prank_v1} to make the formulation as general as possible.} on $X$, we obtain our desired definition for low-prank approximation based on the geometric mean.

\begin{problem}[Low-Prank Approximation]\label{def:low_prank_v1}
	Given a symmetric gauge function $\Phi:~\mathbb{R}^n \to \mathbb{R}$, a sectorial $A\in\mathbb{C}^{n\times n}$, and $0\leq r \leq n$, determine
	\begin{align}\label{eq:low_prank_approx_def_matrix}
	\hat{E}=\argmin_{E}\{\Phi(\phi(E^{-1}AE^{-1}))  \mid E^{-1}AE^{-1}~\text{and}~E~\text{are sectorial},~\prank (E)\leq r\}.
	\end{align}
\end{problem}

\subsection{{Solution to the Low-Prank Approximation Problem for Positive-Imaginary Matrices}}
Similarly to \eqref{eq:truncation_rank}, define the set of all $r$-half-truncation of a sectorial $A\in\mathbb{C}^{n\times n}$ by the sectorial decomposition as
\begin{multline}\label{eq:truncation_prank}
\mathcal{S}_p(A,r):=\{T^*\text{diag}(e^{j\phi_{1}/2},\dots,e^{j\phi_{r}/2},1,\dots,1)T \mid A \text{ is sectorial with}\\ \text{sectorial decomposition } A=T^*\text{diag}(e^{j\phi_{1}},e^{j\phi_{2}},\dots,e^{j\phi_{n}})T \}.
\end{multline}
Analogously to the Schmidt-Mirsky theorem, we have the following result characterizing the optimal solutions to the low-prank approximation in Problem~\ref{def:low_prank_v1} in the case where both $A$ and $X=E^{-1}AE^{-1}$ are constrained to be positive-imaginary.
\begin{theorem}\label{thm:sol_PI}
Given a symmetric gauge function $\Phi:~\mathbb{R}^n \to \mathbb{R}$, $A\in\mathcal{C}[0,\pi)$  and $0\leq r \leq n$, we have
\begin{multline}\label{eq:sol_PI}
\min_E\{\Phi(\phi(E^{-1}AE^{-1})) \mid E^{-1}AE^{-1}\in\mathcal{C}[0,\pi),~E~\text{is sectorial},~\prank (E)\leq r\}\\
=\Phi(0,\dots,0,\phi_{r+1}(A),\dots,\phi_{n}(A)),
\end{multline}
where the optimum is attained on each $\hat{E}\in\mathcal{S}_p(A,r)$. 
\end{theorem}
\begin{proof}
See \ref{app:proof_thm_sol_PI}.
\end{proof}

Note that the above theorem only solves a special case of Problem~\ref{def:low_prank_v1}, where
the matrices are constrained to be positive-imaginary. The general version of Problem~\ref{def:low_prank_v1} may not have the desired solutions as given in \eqref{eq:sol_PI}, which is demonstrated in the following example.
\begin{example}\label{ex:counter_PI}
Consider the following matrix with both positive and negative phases:
	$$A=\begin{bmatrix}
	e^{j\pi/3} & 0  \\
	0 & e^{-j\pi/4}  \\
	\end{bmatrix}, 
	$$
	which violates the positive-imaginary assumption in Theorem~\ref{thm:sol_PI}. Regarding the object matrix $A$, next we show that for some $\Phi(\cdot)$ and $r$, there exists an approximant $\tilde{E}$ so that $\Phi(\phi(\tilde{E}^{-1}A\tilde{E}^{-1}))$ in the left-hand side (LHS) of \eqref{eq:sol_PI} is strictly less than its right-hand side (RHS).
	
	Considering $\Phi(x)=\max_i\{|x_i|\}$ and $r=1$, we construct $\tilde{E}=W^*LW$ with
	$$W=\begin{bmatrix}
		0.6 & 0.3  \\
		0.2 & 0.7  \\
	\end{bmatrix}
	~~\text{and}~~
	L=\begin{bmatrix}
		e^{j\pi/6} & 0 \\
		0 & 1 \\
	\end{bmatrix}.$$
	Clearly, $\tilde{E}$ is sectorial with $\prank(\tilde{E})=1=r$. The numerical range of $\tilde{E}^{-1}A\tilde{E}^{-1}$ is plotted in Figure~\ref{fig:nrange_example}, 
	and as can be seen
	$\tilde{E}^{-1}A\tilde{E}^{-1}$ is {also} sectorial.
	Clearly, neither $A$ nor $\tilde{E}^{-1}A\tilde{E}^{-1}$ are positive-imaginary, which means that Theorem~\ref{thm:sol_PI} cannot be used. 
	Nevertheless, $\tilde{E}$ is a feasible point to the minimization problem in \eqref{eq:low_prank_approx_def_matrix}. Moreover,
	the canonical phases of $\tilde{E}^{-1}A\tilde{E}^{-1}$ are
	$\phi_1 \approx -0.055\pi$ and $\phi_2 \approx -0.22\pi$, where in particular $|\phi_2| < 0.23\pi$. It then follows
	\begin{align*}
	&\min_{E}\{\Phi(\phi(E^{-1}AE^{-1})) \mid E,~E^{-1}AE^{-1}~\text{are sectorial},~\prank (E)\leq r\}\\
	& \leq \Phi(\phi(\tilde{E}^{-1}A\tilde{E}^{-1})) = {|\phi_2(\tilde{E}^{-1}A\tilde{E}^{-1})| < 0.23 \pi} < \frac{\pi}{4}= \Phi(0,-\frac{\pi}{4}){=\text{RHS of}~\eqref{eq:sol_PI}},
	\end{align*}
	which shows that the solutions to Problem~\ref{def:low_prank_v1} {without the assumptions on positive-imaginariness} in general do not possess the structure given in Theorem~\ref{thm:sol_PI}. 

	\begin{figure}
		\centering
		\includegraphics[width=.8\textwidth]{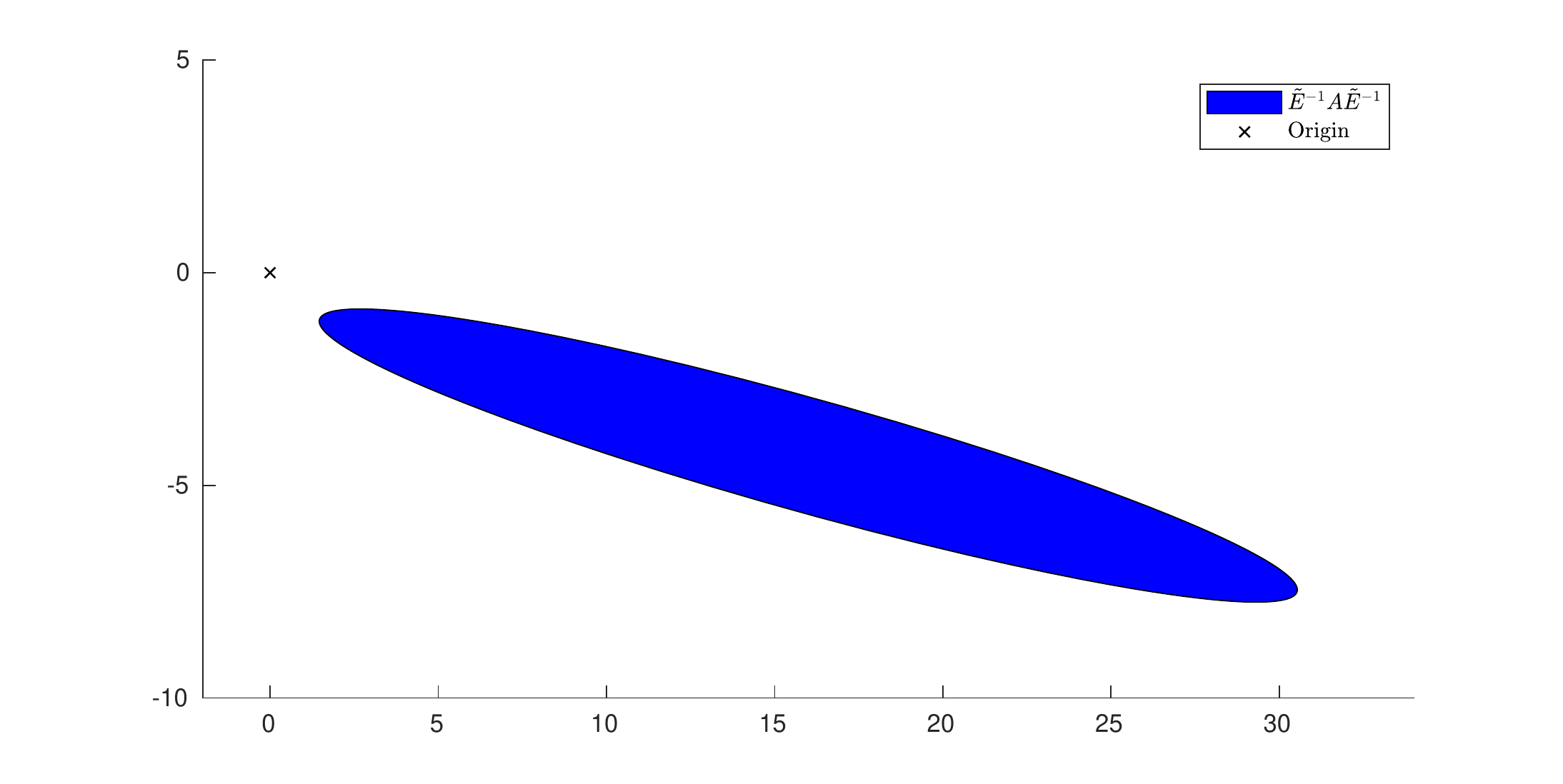}
		\caption{Numerical range of $\tilde{E}^{-1}A\tilde{E}^{-1}$ as in the example.}
		\label{fig:nrange_example}
	\end{figure}
\end{example}

Several interesting and useful corollaries to Theorem~\ref{thm:sol_PI} are as follows. 
\begin{corollary}\label{cor:psi_infty}
	Given $A\in\mathcal{C}[0,\pi)$  and $0\leq r \leq n$, we have
	\begin{align}\label{eq:Sch_Mir_theorem}
	\min_{E}\{\bar{\phi}(E^{-1}AE^{-1})~|~E^{-1}AE^{-1}\in\mathcal{C}[0,\pi),~E~\text{is sectorial},~\prank (E)\leq r\}=\phi_{r+1}(A),
	\end{align}
	{where the optimum is attained on each $\hat{E}\in\mathcal{S}_p(A,r)$.} 
\end{corollary}
\begin{proof}
	The corollary follows by choosing $\Phi(x)=\max\{|x_1|,|x_2|,\dots,|x_n|\}$ in Theorem~\ref{thm:sol_PI}. 
\end{proof}

The following result shows that when the matrices are constrained to be negative-imaginary, the low-prank approximation problem can be similarly solved. 
\begin{corollary}\label{cor:sol_NI}
	Given a symmetric gauge function $\Phi:~\mathbb{R}^n \to \mathbb{R}$, $A\in\mathcal{C}(-\pi,0]$ and $0 \leq r \leq n$, we have
	\begin{multline}\label{eq:sol_NI}
	\min_E\{\Phi(\phi(E^{-1}AE^{-1}))|~E^{-1}AE^{-1}\in\mathcal{C}(-\pi,0],~E~\text{is sectorial},~\prank (E)\leq r\}\\
	=\Phi(\phi_{1}(A),\dots,\phi_{n-r}(A),0,\dots,0),
	\end{multline}
	where the optimum is attained on each $\hat{E}$ with $\hat{E}^{-1}\in\mathcal{S}_p(A^{-1},r)$. 
\end{corollary}
\begin{proof}
The corollary follows by noting that $A$ is positive-imaginary if and only if $A^{-1}$ is negative-imaginary. 
\end{proof}

\subsection{Connections to Rank Optimization Problems}
A natural question about the low-prank approximation problem is whether it can be equivalently transformed into a rank optimization problem. If the answer would be positive, one can solve such a low-prank problem using existing methods for low-rank approximation problems. The following initial result reveals that some low-rank and low-prank approximation problems are related, while
a counterexample shows that this particular result cannot be strengthened further. The latter indicates that the low-rank and low-prank approximation problems are thus in fact fundamentally different.

\begin{theorem}\label{thm:bounds_of_Prank}
Let $A\in\mathbb{C}^{n\times n}$ be sectorial. Then the following statements are true.
\begin{enumerate}
\item [(a)] 
$\displaystyle \text{\rm prank}(A) \; \geq \; \min_{M>0} \text{\rm rank} (A-M) \; = \; \min_{T~\text{is nonsingular}} \text{\rm rank} (T^*AT-I)$;
\item [(b)]There exists a matrix $R\in\mathbb{C}^{n\times n}$ satisfying that $\text{\rm rank}(R)=\text{\rm prank}(A)$ and
$$A-R>0.$$
\end{enumerate}
\end{theorem}

\begin{proof}
Note that (a) follows from (b): the first inequality holds since (b) gives a feasible solution {for which} equality holds, and the second equality holds since all positive definite matrices $M$ can be written as $M = T^{-*}T^{-1}$ for some nonsingular matrix $T$ \cite[Thm.~7.2.7]{horn2013matrix}. To show (b), by the sectorial decomposition we obtain that $A=T^*DT$ for some nonsingular $T$. Let $M=T^*T$ and note that $M>0$. Moreover, for
$R=A-M=T^*(D-I)T$ we have that $\text{rank}(R) = \prank(A)$, which proves the theorem.
\end{proof}

In Theorem~\ref{thm:bounds_of_Prank}(a), the inequality cannot be replaced with equality in general, which can be observed from the following example.

\begin{example}
Let
$$A=\begin{bmatrix}11+j & 0\\0 & 11-j\end{bmatrix} \qquad \text{and} \qquad M=\begin{bmatrix}10 & -\sqrt{2}\\-\sqrt{2} & 10\end{bmatrix}>0.$$
Then $\prank(A) = 2 > 1 = \text{rank}(A-M).$
\end{example}

\section{
A Low Phase-Rank Approximation Problem Based on a Geodesic Distance 
}\label{sec:geodesic}
The set of strictly positive-real matrices
$\mathcal{C}(-\pi/2,\pi/2)$ is a smooth manifold, and in this section we {formulate} an alternative low-prank approximation problem for strictly positive-real matrices based on the geodesic distance introduced in \cite{ringh2020finsler}; for an introduction to smooth manifolds, Riemannian manifolds, and Finsler manifolds see, e.g. \cite{lee2013introduction}, \cite{lee2018introduction}, and \cite{bao2000introduction}, respectively. To introduce the geodesic distance, 
first note that any sectorial matrix has a unique decomposition of the form  \cite[Thm.~3.1]{WANG2020PhaseMath}
\[
A = P U P,
\]
where $U \in \mathcal{U}_n$ and $P \in \mathcal{P}_n := $ the set of Hermitian positive definite matrices. {Such a decomposition is called the symmetric polar decomposition.} Since the phases of a matrix are invariant under congruence, the phases of $A$ are the same as the phases of $U$; in particular, this means that for $A \in \mathcal{C}(-\pi/2,\pi/2)$ we have that $U \in  \mathcal{U}_n^a := \mathcal{U}_n  \cap \mathcal{C}(-\pi/2,\pi/2)$, i.e., the set of strictly positive-real unitary matrices.
In fact, by means of the symmetric polar decomposition, $\mathcal{C}(-\pi/2,\pi/2)$ is diffeomorphic to $\mathcal{P}_n \times \mathcal{U}_n^a$ \cite[Thm.~3.2]{ringh2020finsler}. Based on this decomposition,  in \cite{ringh2020finsler} a family of Finsler structures were introduced on $\mathcal{C}(-\pi/2,\pi/2)$ and the corresponding geodesics and geodesic distances were characterized (cf. \cite{okada1982minkowskian}). For brevity and clarity, in the following we will focus on one subfamily, however the results can be extended to the entire family. More specifically, 
for $A,B \in \mathcal{C}(-\pi/2,\pi/2)$ with {the} corresponding symmetric polar decompositions $A=P_AU_AP_A$ and $B=P_BU_BP_B$, we consider the geodesic distance between $A$ and $B$ given by \cite[{Thm.~4.4}]{ringh2020finsler}
\begin{align}\label{eq:geodesic}
\delta_{\mathcal{C}(-\pi/2,\pi/2)}(A,B):=\sqrt{\|\log(P_A^{-1}P_B^2P_A^{-1})\|_{\Phi_1}^{2}+\|\log(U_A^*U_B)\|_{\Phi_2}^2},
\end{align}
where $\| \cdot \|_{\Phi_k}$ is the unitarily invariant norm induced by the smooth symmetric gauge function $\Phi_k$, for $k = 1,2$.%
\footnote{A symmetric gauge function $\Phi$ is called smooth if it is smooth outside of the origin. This is equivalent to that the induced norm $\| \cdot \|_{\Phi}$ is smooth outside of the origin, cf. \cite[Thm.~8]{lewis1996group}.}
Using this distance, we consider the following alternative low-prank approximation problem.
\begin{problem}\label{def:low_prank_v_geodesic}
Given two symmetric gauge functions $\Phi_k:~\mathbb{R}^n \to \mathbb{R}$, $k = 1,2$, a matrix $A \in \mathcal{C}(-\pi/2,\pi/2)$, and $0 \leq r \leq n$, find
\begin{equation}\label{eq:phase_rank_approx}
\begin{aligned}
\inf_{A_r \in \mathcal{C}(-\pi/2,\pi/2)} & \quad \delta_{\mathcal{C}(-\pi/2,\pi/2)}(A, A_r) \\
\text{subject to} & \quad \prank(A_r) \leq r.
\end{aligned}
\end{equation}
\end{problem}

The solution to this problem is given by the following theorem.
{For $A\in\mathcal{C}(-\pi/2,\pi/2)$, let $[\tilde{\phi}_k(A)]_{k = 1}^n$ be its phases satisfying the following order:}
$${|\tilde{\phi}_1(A)| \geq |\tilde{\phi}_2(A)| \geq \cdots \geq |\tilde{\phi}_n(A)|.}$$

\begin{theorem}\label{thm:geo_dist_approx}
The infimum in \eqref{eq:phase_rank_approx} is attained by $\hat{A}_r = P \hat{U}_r P$, where $A = PUP$ is the symmetric polar decomposition of $A$ and $\hat{U}_r$ is obtained from $U$ by setting the $n-r$ phases of $U$ with smallest absolute value equal to $0$. That is, for $U = V^*DV$ being a diagonalization of $U$ with $D = \text{\rm diag}(e^{j\tilde{\phi}_1(U)}, \dots, e^{j\tilde{\phi}_n(U)})$, the optimal $\hat{U}_r$ is given by $\hat{U}_r = V^*D_r V$ with
$$D_r = \text{\rm diag}(e^{j\tilde{\phi}_1({U})},\dots,e^{j\tilde{\phi}_r({U})},0,\dots,0).$$
\end{theorem}

\begin{proof}
See \ref{app:proof_bf_geo_dist_approx}
\end{proof}

Note that a minimizer to \eqref{eq:phase_rank_approx} might not be unique. In fact, it is unique if and only if $|\tilde{\phi}_{r}(U)| > |\tilde{\phi}_{r+1}(U)|$.
Moreover, the geodesic distance \eqref{eq:geodesic} can be extended to sets of sectorial matrices {whose numerical ranges are bounded in a common half plane without intersecting the negative real axis,} i.e., to sets $\mathcal{C}(-\pi/2+\alpha, \pi/2 + \alpha)$ for any $\alpha \in (-\pi/2, \pi/2)$ \cite[{Rem.~3}]{ringh2020finsler}. In particular, for $A,B  \in \mathcal{C}(-\pi/2+\alpha,\pi/2+\alpha)$, the distance is computed by using the rotation $e^{-j \alpha}$ to make them positive real, and is hence given by 
\begin{align*}
\delta_{\mathcal{C}(-\pi/2+\alpha,\pi/2+\alpha)}(A, B) & :=\delta_{\mathcal{C}(-\pi/2,\pi/2)}(e^{-j\alpha}A, e^{-j\alpha}B)  \\
& = \sqrt{\|\log(P_A^{-1}P_B^2P_A^{-1})\|_{\Phi_1}^{2}+\|\log(U_A^*U_B)\|_{\Phi_2}^2},
\end{align*}
and the algebraic expression is thus the same as in \eqref{eq:geodesic}. By adapting the arguments in the proof of Theorem~\ref{thm:geo_dist_approx}, we get the following corollary.

\begin{corollary}\label{cor:geo_dist}
	Let $A \in \mathcal{C}(-\pi/2+\alpha,\pi/2+\alpha)$ with $\alpha\in(-\pi/2,\pi/2)$. The infimum to
		\begin{align*}
		\inf_{A_r\in \mathcal{C}(-\pi/2+\alpha,\pi/2+\alpha)} & \quad \delta_{\mathcal{C}(-\pi/2+\alpha,\pi/2+\alpha)}(A, A_r) \\
		\text{\rm subject to} & \quad \prank(A_r) \leq r,
		\end{align*}
		is attained on each $\hat{A}_r$ as described in Theorem~\ref{thm:geo_dist_approx}.
\end{corollary}

\begin{remark}\label{rem:extension_half_closed}
Note that the distance \eqref{eq:geodesic} can be directly extended to i) symmetric gauge functions that are not smooth, and ii) sets of sectorial matrices with half-closed angular intervals, i.e., $\mathcal{C}(-\pi/2 + \alpha, \pi/2+\alpha]$ and $\mathcal{C}[-\pi/2 + \alpha, \pi/2+\alpha)$ for $\alpha \in (-\pi/2, \pi/2)$. In this case, the manifold structure (might) be destroyed, {but} the solution to the approximation problem corresponding to \eqref{eq:phase_rank_approx} remains the same.
\end{remark}

\subsection{Connections to the Original Formulation}
The alternative formulation of the low-prank approximation problem in terms of the geodesic distance is related to the original formulation in Problem~\ref{def:low_prank_v1}, even though the two formulations are in general not equivalent. Nevertheless, if we impose additional requirements, such as positive-imaginariness, then from Theorem~\ref{thm:sol_PI} and Corollary \ref{cor:geo_dist} we see that the minimum approximation error, i.e., the optimal value of the objective function, is in fact the same in both problems. 
Moreover, if we also require the matrix $A \in \mathcal{C}[0,\pi) $ to be unitary, then the following result shows that both the problems share the same set of solutions.
\begin{theorem}\label{thm:relation}
	Let $\Phi:~\mathbb{R}^n \to \mathbb{R}$ be a symmetric gauge function, $A\in\mathcal{C}[0,\pi) \cap  \mathcal{U}_n$, and $0\leq r \leq n$. Then the set of optimal solution to Problem~\ref{def:low_prank_v1} is the set $ \mathcal{S}_p(A,r)$, which consists of only unitary matrices.
	Moreover, let
	\begin{equation}\label{eq:thmrelation_1}
	\hat{\mathcal{B}} := \argmin_{B}\{\delta_{\mathcal{C}[0, \pi)}(A, B) |~B \in \mathcal{C}[0,\pi),~\prank(B)\leq r\}.
	\end{equation}
	 Then $\mathcal{S}_p(A,r) =  \hat{\mathcal{B}}^{\frac{1}{2}}$, and the optimal value for both Problem~\ref{def:low_prank_v1} and \eqref{eq:thmrelation_1} is given by $\Phi(0,\dots,0,\phi_{r+1}(A),\dots,\phi_{n}(A))$.
\end{theorem}
\begin{proof}
By Theorem~\ref{thm:sol_PI} the set $\mathcal{S}_p(A,r)$ gives the optimal solutions to Problem~\ref{def:low_prank_v1} and the corresponding optimal value is $\Phi(0,\dots,0,\phi_{r+1}(A),\dots,\phi_{n}(A))$.
Moreover, if $A$ is unitary then all sectorial decompositions $A = T^*\text{diag}(e^{j\phi_{1}},e^{j\phi_{2}},\dots,e^{j\phi_{n}})T$ of $A$ are with unitary matrices $T$. Therefore, $\mathcal{S}_p(A,r)$ consists of only unitary matrices.
Next, by Corollary~\ref{cor:geo_dist} and Remark~\ref{rem:extension_half_closed}, the optimal value to \eqref{eq:thmrelation_1} is $\Phi(0,\dots,0,\phi_{r+1}(A),\dots,\phi_{n}(A))$ and we have that $\hat{\mathcal{B}} \subset \mathcal{U}_n$. In fact, a direct calculation shows that for any $E \in \mathcal{S}_p(A,r)$, $E^2 \in \hat{\mathcal{B}}$, and for any $B \in \hat{\mathcal{B}}$, $B^{\frac{1}{2}} \in \mathcal{S}_p(A,r)$. Hence, $\mathcal{S}_p(A,r) = \hat{\mathcal{B}}^{\frac{1}{2}}$, which proves the theorem.
\end{proof}

To see that the two problem formulations are not equivalent in general, recall Example~\ref{ex:counter_PI}. By taking $\Phi(x)=\Phi_1(x)=\max_i\{|x_i|\}$ and $r=1$ in both the problems formulations, we have
\begin{align*}
&\min_{E}\{\Phi(\phi(E^{-1}AE^{-1})) \mid E,~E^{-1}AE^{-1}~\text{are sectorial},~\prank (E)\leq r\}\leq \Phi(\phi(\tilde{E}^{-1}A\tilde{E}^{-1}) \\
& = {|\phi_2(\tilde{E}^{-1}A\tilde{E}^{-1})|} \! < \! \Phi(0,-\frac{\pi}{4}) =\inf_{A_r} \{ \delta_{\mathcal{C}(-\pi/2,\pi/2)}(A, A_r) \mid A_r\in\mathcal{C}(-\pi/2,\pi/2),~ \prank(A_r) \! \leq \! r\}.
\end{align*}
This example thus illustrates {a} difference between the two formulations. We expect both of them to be useful in different application settings.

\section{Conclusion and Future Directions}\label{sec:conc}
{This paper investigates a low phase-rank approximation problem, serving as a counterpart to the well-studied low rank approximation problem. The problem is formulated in two forms --- one in terms of matrix geometric means and the other in terms of geodesic distances {between} sectorial matrices.  Characterizations of solutions to the low-prank approximation problems are obtained, which are presented with a similar {flavor to that of} the well-known Schmidt-Mirsky theorem. }

As for future research, we may explore more connections of the low-prank approximation to real-world applications, such as model reduction, data compression and image processing.  On the other hand, there are technical issues that are worth further exploration. In particular, as currently we focus on matrices that are sectorial within the same sector, we may try to remove such constraints and study the approximation problem among matrices that are sectorial in any sector. {Another direction that is worth exploration lies in extending matrices to linear operators over infinite-dimensional linear spaces.} 

\section*{Acknowledgement}
The authors would like to thank Wei Chen, Dan Wang and Chao Chen for valuable discussions. 

This work was supported in part {by the Guangdong Science and Technology Department, China, under the project No.~2019B010117002, the Research Grants Council of Hong Kong Special Administrative Region, China, under the project GRF 16200619}, and the Knut and Alice Wallenberg foundation, Stockholm, Sweden, under grant KAW 2018.0349.

\appendix
\section{Proof of Theorem~\ref{thm:sol_PI}}\label{app:proof_thm_sol_PI}
Denote a special class of symmetric gauge functions, called the Ky-Fan $k$-norms \cite{kyfan1955norm}, $k=1,2,\dots,n$, by
$$\Phi_k(x):=\text{sum of the largest}~k~\text{terms in}~ \{|x_1|,|x_2|,\dots,|x|_n\}.$$
For $A\in\mathcal{C}[0,\pi)$, define the family of the ``Ky-Fan'' phase values (i.e., partial sum of the canonical phases from the larger side) as
$$\psi_m(A):=\Phi_m(\phi(A)),~m=1,2,\dots, n.$$
By \cite[Lemma~4.4]{WANG2020PhaseMath}, we obtain that 
\begin{align}\label{eq:sum_p_matrix}\psi_m(A)=\max_{X\in\mathbb{C}^{n\times m}}\sum_{k=1}^m\angle\lambda_k(X^*AX),~m=1,2,\dots, n.\end{align}
\begin{lemma}\label{lem:sectorial_perturb}
	Let $\theta\in(-\pi,0]$, $A,B\in\mathcal{C}[\theta,\theta+\pi)$ and $\underline{\phi}(A)\geq \bar{\phi}(B)$. Then it holds that $$\psi_m(A+B)\geq\psi_m(B)$$
	for all $m=1,2,\dots,n$.
\end{lemma}
\begin{proof}
	Since $A,B\in\mathcal{C}[\theta,\theta+\pi)$, {we have that $A+B\in\mathcal{C}[\theta,\theta+\pi)$ \cite[p.~2]{zhang2015matrix}, \cite[Thm.~7.1]{WANG2020PhaseMath}}. It follows by the sectorial decomposition that there exists an $\hat{X}\in\mathbb{C}^{n\times m}$  such that $\hat{X}^*B\hat{X}=\text{diag}(e^{j\phi_1(B)},e^{j\phi_2(B)},\dots,e^{j\phi_m(B)})=:\Lambda$. Then by definition, it holds that
	\begin{multline}\label{eq:pf_sectorial_perturb}
	\psi_m(A+B)=\max_{X\in\mathbb{C}^{n\times m}}\sum_{k=1}^m\angle\lambda_k(X^*(A+B)X)\geq\sum_{k=1}^m\angle\lambda_k(\hat{X}^*A\hat{X}+\Lambda)\\
	= \psi_m(B)+\sum_{k=1}^m\angle\lambda_k(\hat{X}^*A\hat{X}\Lambda^{-1}+I).
	\end{multline}
	Clearly, both $\hat{X}^*A\hat{X}$ and $\Lambda^{-1}$ are sectorial, yielding that
	\begin{multline*}\pi>\bar{\phi}(A)-\underline{\phi}(B)\geq\bar{\phi}(\hat{X}^*A\hat{X})+\bar{\phi}(\Lambda^{-1})\geq \angle\lambda_k(\hat{X}^*A\hat{X}\Lambda^{-1})\geq \underline{\phi}(\hat{X}^*A\hat{X})+\underline{\phi}(\Lambda^{-1})\\
	\geq\underline{\phi}(A)- \bar{\phi}(B)\geq 0,\end{multline*}
	for all $k=1,2,\dots,m$. Since a shift of eigenvalues of $\hat{X}^*A\hat{X}\Lambda^{-1}$ along the positive real axis will keep their phases in $[0,\pi)$, it must hold that
	$$\sum_{k=1}^m\angle\lambda_k(\hat{X}^*A\hat{X}\Lambda^{-1}+I)\geq 0,~k=1,2,\dots,m.$$
	Plugging the inequality into (\ref{eq:pf_sectorial_perturb}) completes the proof.
\end{proof}

\begin{proof}[Proof of Theorem~\ref{thm:sol_PI}]
	The case when $r=0$ or $r=n$ is trivially true. Next we let $r\in[1,n-1]$. Applying the sectorial decomposition on $A$ yields that
	$$A=T^*DT,$$
	where $T$ is nonsingular, $D=\text{diag}(e^{j\phi_1(A)},e^{j\phi_2(A)},\dots,e^{j\phi_n(A)})$ and $\phi_k(A)\in[0,\pi)$, $k=1,2,\dots,n$. 
	Consider ${\hat{E}}=T^*\Lambda T\in\mathcal{S}_p(A,r)$ with $$\Lambda:=\text{diag}(e^{j\phi_1(A)/2},\dots,e^{j\phi_{r}(A)/2},1,\dots,1).$$ 
	It is straightforward to verify that
	$${\hat{E}}^{-1}A{\hat{E}}^{-1}=T^{-1}\text{diag}(1,\dots,1,e^{j\phi_{r+1}(A)},\dots,e^{j\phi_n(A)})T^{-*},$$
	{which} is positive-imaginary {since $A$ is positive-imaginary. This}
	implies that
	$$\Phi(\phi({\hat{E}}^{-1}A{\hat{E}}^{-1}))=\Phi(0,\dots,0,\phi_{r+1}(A),\dots,\phi_n(A)),$$
	{and hence that}
	\begin{multline*}
	\min_E\{\Phi(\phi(E^{-1}AE^{-1}))|~E^{-1}AE^{-1}\in\mathcal{C}[0,\pi),~E~\text{is sectorial},~\prank (E)\leq r\}\\\leq \Phi(0,\dots,0,\phi_{r+1}(A),\dots,\phi_n(A))
	\end{multline*}
	since {${\hat{E}}$} is a feasible point to the minimization problem. Therefore,
	it suffices to show that for all feasible $E$, it holds that
	\begin{align}\label{eq:pf_PI2}
	\Phi(\phi(E^{-1}AE^{-1})\geq \Phi(0,\dots,0,\phi_{r+1}(A),\dots,\phi_n(A)).
	\end{align}
	Furthermore, by 
	\cite[Thm.~3.7]{stewart1990matrix}
	we obtain that showing \eqref{eq:pf_PI2} is equivalently to showing that
	\begin{align}\label{eq:pf_PI3}
	\psi_m(E^{-1}AE^{-1})\geq \Phi_m(0,\dots,0,\phi_{r+1}(A),\dots,\phi_n(A))=\sum_{k=r+1}^{\min\{n,r+m\}}\phi_k(A),
	\end{align}
	for $m=1,2,\dots,n$.
	
	To this end, let $E\in\mathbb{C}^{n\times n}$ be an arbitrary sectorial matrix such that $E^{-1}AE^{-1}$ is positive-imaginary and $\prank(E) \leq r$. The sectorial decomposition yields that
	$$E=S^*LS,$$
	where $S$ is nonsingular, $L=\text{diag}(L_0,I)$, and $L_0=\text{diag}(e^{j\phi_1(E)},e^{j\phi_2(E)},\dots,e^{j\phi_r(E)})$. Let $W=TS^{-1}$, and partition $D$ and $W$ conformably with $L$ as
	$$D=\begin{bmatrix}
	D_1 & 0 \\
	0 & D_2 \\
	\end{bmatrix}~\text{and}~
	W=\begin{bmatrix}
	W_{11} & W_{12} \\
	W_{21} & W_{22} \\
	\end{bmatrix},
	$$
	respectively. For $k\in[1,n-r]$, we have
	\begin{equation}\label{eq:pf_SchMir_relax}
	\begin{aligned}
	\phi_k(E^{-1}AE^{-1})&=\phi_k(S^{-1}L^{-1}W^*DWL^{-1}S^{-*})=\phi_k(L^{-1}W^*DWL^{-1})\\
	&=\phi_k\left(\begin{bmatrix}
	\star & \star \\
	\star & W_{12}^*D_1W_{12}+W_{22}^*D_2W_{22} \\
	\end{bmatrix}
	\right)\\
	&\geq \phi_k(W_{12}^*D_1W_{12}+W_{22}^*D_2W_{22}),
	\end{aligned}
	\end{equation}
	where the inequality follows from the interlacing property of the compressed sectorial matrix; see \cite[Sec.~4]{WANG2020PhaseMath} for details.
	
	Since $W_{12}^*D_1W_{12}+W_{22}^*D_2W_{22}$ is a compression of the sectorial matrix $L^{-1}W^*DWL^{-1}$, it is sectorial as well. If either $W_{12}\in\mathbb{C}^{r\times (n-r)}$ or $W_{22}\in\mathbb{C}^{(n-r)\times (n-r)}$ is not full 
	rank, we can perturb them slightly to make them full rank. Since the phases of a sectorial matrix
	are continuous on its elements, the {variation of the phases of $W_{12}^*D_1W_{12}+W_{22}^*D_2W_{22}$} due to the perturbations can be made arbitrarily small.
	{Therefore}, without loss of generality, we may assume both $W_{12}$ and $W_{22}$ are full
	rank, and hence the square matrix $W_{22}^*D_2W_{22}$ is nonsingular.
	
	Now, we first consider the case when $r<n-r$. To this end, let
	$$\tilde{D}_1:=\begin{bmatrix}
	D_1 & \\ & e^{j\underline{\phi}(D_1)}I \\
	\end{bmatrix}\in\mathbb{C}^{(n-r)\times(n-r)}~\text{and}~\tilde{W}_{12}:=\begin{bmatrix}
	W_{12} \\ \Delta \\
	\end{bmatrix}\in\mathbb{C}^{(n-r)\times(n-r)},
	$$
	where $\Delta$ is such that $\tilde{W}_{12}$ is nonsingular; such an extension is always possible since $W_{12}$ is nonsingular. Moreover, clearly $\tilde{W}_{12}^*\tilde{D}_1\tilde{W}_{12}$ is nonsingular. Now, by the fact that $W_{12}^*D_1W_{12}+W_{22}^*D_2W_{22}$ is sectorial, we know that
	$$\tilde{W}_{12}^*\tilde{D}_1\tilde{W}_{12}+W_{22}^*D_2W_{22}=W_{12}^*D_1W_{12}+W_{22}^*D_2W_{22}+e^{j\underline{\phi}(D_1)}\Delta^*\Delta$$
	is sectorial as long as $\|\Delta\|$ is sufficiently small. As a result, it holds that
	\begin{align*}
	\psi_{m}(W_{12}^*D_1W_{12}+W_{22}^*D_2W_{22})\geq \psi_{m}(\tilde{W}_{12}^*\tilde{D}_1\tilde{W}_{12}+W_{22}^*D_2W_{22})-\epsilon_{\Delta},
	\end{align*}
	where $\epsilon_{\Delta}>0$ can be chosen to be arbitrarily small by bounding $\Delta$ properly.
	{Next, consider the case $r\geq n-r$, in which case $W_{12}$ is a square matrix ($r = n-r$) or tall matrix ($r > n-r$). In any case, letting $\tilde{W}_{12}=W_{12}$ and $\tilde{D}_1=D_1$,} the matrix $\tilde{W}_{12}^*\tilde{D}_1\tilde{W}_{12}$ is nonsingular as well. As a result, for all $r\in[0,n]$, both $\tilde{W}_{12}^*\tilde{D}_1\tilde{W}_{12}$ and $W_{22}^*D_2W_{22}$ are sectorial with phases in $[\underline{\phi}(A),\bar{\phi}(A)]\subset[\theta,\theta+\pi)$.
	
	Finally, to show that \eqref{eq:pf_PI3} holds for $m = 1, \ldots, n$, we again split it into two cases. First, for any $m\in[1,n-r]$, noting that
	$$\underline{\phi}(\tilde{W}_{12}^*\tilde{D}_1\tilde{W}_{12})\geq\underline{\phi}(\tilde{D}_1)=\underline{\phi}({D}_1)=\phi_r(A)\geq \phi_{r+1}(A)=\bar{\phi}(D_2)=\bar{\phi}(W_{22}^*D_2W_{22})$$
	and by Lemma~\ref{lem:sectorial_perturb}, we have
	\begin{multline*}
	\psi_m(E^{-1}AE^{-1}) \geq \psi_{m}(W_{12}^*D_1W_{12}+W_{22}^*D_2W_{22})\geq \psi_{m}(\tilde{W}_{12}^*\tilde{D}_1\tilde{W}_{12}+W_{22}^*D_2W_{22})-\epsilon_{\Delta}\\
	\geq \psi_{m}(W_{22}^*D_2W_{22})-\epsilon_{\Delta} = \sum_{k=r+1}^{r+m}\phi_{k}(A)-\epsilon_{\Delta}.
	\end{multline*}
	Since $\epsilon_{\Delta}$ is independent of $\phi_{k}(A)$, $k=1,2,\dots,n$, and can be made arbitrarily small, we have
	\begin{align}\label{eq:pf_SchMir_smallsum}
	\psi_m(E^{-1}AE^{-1}) \geq  \sum_{k=r+1}^{r+m}\phi_{k}(A).
	\end{align}
	{Next, for all $m\in(n-r,n]$, noting that $E^{-1}AE^{-1}$ is positive-imaginary and hence has nonnegative phases, we obtain that $\psi_m(E^{-1}AE^{-1}) \geq \psi_{n-r}(E^{-1}AE^{-1})$. By using (\ref{eq:pf_SchMir_smallsum}), we therefore have}
	\begin{align}\label{eq:pf_Sch_Mir_F}
	\psi_m(E^{-1}AE^{-1}) \geq \psi_{n-r}(E^{-1}AE^{-1}) \geq \sum_{k=r+1}^{n}\phi_{k}(A).
	\end{align}
	On the other hand, it is straightforward to verify that
	\begin{align}\label{eq:pf_PI4} \Phi_m(0,\dots,0,\phi_{r+1}(A),\dots,\phi_n(A))=\sum_{k=r+1}^{\min\{n,r+m\}}\phi_k(A).
	\end{align}
	The combination of (\ref{eq:pf_SchMir_smallsum}), (\ref{eq:pf_Sch_Mir_F}) and (\ref{eq:pf_PI4}) shows \eqref{eq:pf_PI3}, which completes the proof.
\end{proof}

\section{Proof of Theorem~\ref{thm:geo_dist_approx} }\label{app:proof_bf_geo_dist_approx}
For the proof of Theorem~\ref{thm:geo_dist_approx} we need the concept of \emph{majorization}. To this end, for $a \in \mR^n$ by $a^{\downarrow}$ we denote the vector obtained by sorting $a$ in a nonincreasing order, i.e., $a^{\downarrow} = [a_k^\downarrow]_{k=1}^n = [a_{\alpha(k)}]_{k=1}^n$ where $\alpha$ is a permutation on $\{ 1,\ldots, n \}$ such that $a_1^\downarrow \geq a_2^\downarrow \geq \ldots \geq a_n^{\downarrow}$. For two vectors $a,b \in \mR^n$, we write $a \prec_w b$ if $\sum_{k = 1}^{\ell} a_k^\downarrow \leq \sum_{k = 1}^{\ell} b_k^\downarrow$ for $\ell = 1, \ldots, n$; in this case we say that $a$ is weakly submajorized by $b$ \cite[p.~12]{marshall2011inequalities}. In fact, $\prec_w$ is a preordering on $\mR^n$, and on the equivalence classes defined by identifying ``$a = b$'' if and only if $a^{\downarrow} = b^{\downarrow}$ (i.e., if and only if $a \prec_w b$ and $b \prec_w a$), it is a partial ordering  \cite[p.~19]{marshall2011inequalities}.

With $\varphi(\cdot)$ we denote the angle of the eigenvalues of a matrix, i.e., $\varphi(\cdot) := \angle \lambda(\cdot)$. In particular, note that for a unitary sectorial matrix $U$ we have that $\varphi(U) = \phi(U)$.

\begin{proof}[Proof of Theorem~\ref{thm:geo_dist_approx}]
If $A$ has phase rank $\prank(A) \leq r$, then the infimum is $0$ and obtained at $\hat{A}_r = A$. Without loss of generality, in the remaining we will therefore assume that $\prank(A)>r$.
To this end, let $A = PUP$ and $A_r = P_rU_rP_r$ be the corresponding symmetric polar decompositions, and note that the phases of $A$ and $A_r$ depend only on the phases of $U$ and $U_r$, respectively.
Using the explicit form of the geodesic distance \eqref{eq:geodesic},
problem \eqref{eq:phase_rank_approx} can be rewritten as
\begin{subequations}\label{eq:phase_rank_approx_PandU}
\begin{align}
\inf_{\substack{P_r \in \mathcal{P}_n \\ U_r \in \mathcal{U}_n^a}} & \quad  \sqrt{\|\log(P^{-1}P_r^2P^{-1})\|_{\Phi_1}^{2}+\|\log(U^*U_r)\|_{\Phi_2}^2} \label{eq:phase_rank_approx_PandU_cost} \\
\text{subject to} & \quad \prank(U_r) \leq r,
\end{align}
\end{subequations}
which separates in the two variables $P_r$ and $U_r$. 
Moreover, the global minimum over $P_r$ is attained for $\hat{P}_r = P$, which means that we are left with the problem to minimize $\| \log(U^{-1}U_r) \|_{\Phi_2}$ over all $U_r \in \mathcal{U}_n^a$ such that $\prank(U_r) \leq r$.
Using the same arguments as in the proof of \cite[Thm.~4.4]{ringh2020finsler},
we therefore have that an equivalent problem to \eqref{eq:phase_rank_approx} is
\begin{subequations}\label{eq:phase_rank_approx_U}
\begin{align}
\inf_{U_r \in \mathcal{U}_n^a} & \quad \Phi_{2}( |\varphi(U^{-1}U_r )|) \\
\text{subject to} & \quad \prank(U_r) \leq r.
\end{align}
\end{subequations}
To show that the optimal solution to \eqref{eq:phase_rank_approx_U} is $\hat{U}_r$ for all symmetric gauge functions $\Phi_2$, we {equivalently show} that $|\varphi(U^{-1} \hat{U}_r )| \prec_w  |\varphi(U^{-1} U_r )|$ for all $U_r \in \mathcal{U}_n^a$ with $\prank(U_r) \leq r$
\cite[Thm.~4]{fan1951maximum}; see also, e.g., \cite[Thm.~1]{mirsky1960symmetricgaugefunction}, \cite[Sec.~3.5]{horn1994topics}, \cite[Prop.~4.B.6]{marshall2011inequalities}.
To prove the latter, we note that by \cite[Thm.~2]{chau2011metrics} we have that $ | |\varphi(U^{-1})|^{\downarrow} - |\varphi(U_r)|^\downarrow | \prec_w |\varphi(U^{-1}U_r)|^{\downarrow}$. 
Moreover, it is easily verified that $| |\varphi(U^{-1})|^{\downarrow} - |\varphi(\hat{U}_r)|^{\downarrow} |  =|\varphi(U^{-1}\hat{U}_r)|^{\downarrow}$. Therefore, if we can show that 
$\hat{U}_r$ is a minimum of
\begin{align*}
\min_{\prec_w} & \quad  | |\varphi(U)|^{\downarrow} - |\varphi(U_r)|^\downarrow | \\
\text{subject to} & \quad U_r \in \mathcal{U}_n^a, \quad \prank(U_r) \leq r,
\end{align*}
where $\min_{\prec_w}$ is minimizing with respect to the preordering $\prec_w$, then for all $U_r \in \mathcal{U}_n^a$ with $\prank(U_r) \leq r$ we have that
\[
|\varphi(U^{-1}\hat{U}_r)|^{\downarrow} = | |\varphi(U^{-1})|^{\downarrow} - |\varphi(\hat{U}_r)|^{\downarrow} |  \prec_w | |\varphi(U)|^{\downarrow} - |\varphi(U_r)|^\downarrow | \prec_w |\varphi(U^{-1}U_r)|^{\downarrow}.
\]
Moreover, {by the} transitivity of preorders, we obtain that $|\varphi(U^{-1}\hat{U}_r)| \prec_w |\varphi(U^{-1}U_r)|$, which is the desired inequality that would prove the theorem.
To this end, let 
\[
\script{D}_{r} := \{ x \in \mR^n \mid x_1 \geq x_2 \geq \ldots \geq x_n, \;  0 \leq x_k < \pi/2, \text{ and at most $r$ of the $x_k$s are nonzero}  \},
\]
and note that for all $x \in \script{D}_r$ there is at least one $U_r \in \mathcal{U}_n^a$ with $\prank(U_r) \leq r$ such that $|\phi(U_r)|^{\downarrow} = |\varphi(U_r)|^{\downarrow} = x$. We can therefore equivalently consider $\min_{\prec_w}| |\varphi(U^{-1})|^{\downarrow} - x|$ subject to $x \in \script{D}_r$. For the latter problem, a minimizer is given by making the first $r$ components of $|\varphi(U^{-1})|^{\downarrow} - x$ equal to zero, i.e., by taking $x_k = |\tilde{\phi}_k(U)|$ for $k = 1, \ldots r$ and $x_k =0$ for $k = r+1, \ldots, n$. This means that $x = |\phi(\hat{U}_r)|^{\downarrow}$ is a minimizer, and hence proves the theorem. 
\end{proof}

\bibliography{prank}

\end{document}